\documentclass{elsarticle}
\usepackage{mathtools}
\usepackage{amssymb,amsthm}
\usepackage{mathrsfs,comment}
\usepackage[usenames,dvipsnames]{color}
\usepackage[normalem]{ulem}
\usepackage{url}
\usepackage[all,arc,2cell]{xy}
\UseAllTwocells
\usepackage{enumerate}
\usepackage{hyperref}  
\hypersetup{%
  bookmarksnumbered=true,%
  colorlinks=true,%
  linkcolor=blue,%
  citecolor=blue,%
  filecolor=blue,%
  menucolor=blue,%
  urlcolor=blue,%
  pdfnewwindow=true,%
  pdfstartview=FitBH}

\def\makeautorefname#1#2{\expandafter\def\csname#1autorefname\endcsname{#2}}
\def\equationautorefname~#1\null{(#1)\null}
\makeautorefname{footnote}{footnote}%
\makeautorefname{item}{item}%
\makeautorefname{figure}{Figure}%
\makeautorefname{table}{Table}%
\makeautorefname{part}{Part}%
\makeautorefname{appendix}{Appendix}%
\makeautorefname{chapter}{Chapter}%
\makeautorefname{section}{Section}%
\makeautorefname{subsection}{Section}%
\makeautorefname{subsubsection}{Section}%
\makeautorefname{theorem}{Theorem}%
\makeautorefname{thm}{Theorem}%
\makeautorefname{cor}{Corollary}%
\makeautorefname{lem}{Lemma}%
\makeautorefname{prop}{Proposition}%
\makeautorefname{pro}{Property}
\makeautorefname{conj}{Conjecture}%
\makeautorefname{defn}{Definition}%
\makeautorefname{notn}{Notation}
\makeautorefname{notns}{Notations}
\makeautorefname{rem}{Remark}%
\makeautorefname{quest}{Question}%
\makeautorefname{exmp}{Example}%
\makeautorefname{ax}{Axiom}%
\makeautorefname{claim}{Claim}%
\makeautorefname{ass}{Assumption}%
\makeautorefname{asss}{Assumptions}%
\makeautorefname{con}{Construction}%
\makeautorefname{prob}{Problem}%
\makeautorefname{warn}{Warning}%
\makeautorefname{obs}{Observation}%
\makeautorefname{conv}{Convention}%

\newtheorem{thm}{Theorem}[section]

\newtheorem{prop}{Proposition}[section]
\newtheorem{lem}{Lemma}[section]

\theoremstyle{definition}
\newtheorem{defn}{Definition}[section]

\makeatletter
\let\c@obs=\c@thm
\let\c@cor=\c@thm
\let\c@prop=\c@thm
\let\c@lem=\c@thm
\let\c@prob=\c@thm
\let\c@con=\c@thm
\let\c@conj=\c@thm
\let\c@defn=\c@thm
\let\c@notn=\c@thm
\let\c@notns=\c@thm
\let\c@exmp=\c@thm
\let\c@ax=\c@thm
\let\c@pro=\c@thm
\let\c@ass=\c@thm
\let\c@warn=\c@thm
\let\c@rem=\c@thm
\let\c@sch=\c@thm
\let\c@equation=\c@thm
\numberwithin{equation}{section}
\makeatother

\DeclareMathOperator{\ex}{ex}
\DeclareMathOperator{\FR}{FR}

\begin{document}

\begin{frontmatter}

\title{Non-jumping densities of 3-uniform hypergraphs}

\author[uchicago]{Vaughn Komorech}
\ead{vaughnkomorech@gmail.com}

\address[uchicago]{Department of Mathematics,
University of Chicago,
Chicago, IL 60637, USA}

\begin{abstract}
    A density $\alpha\in [0, 1)$ is a jump for $r$ if there is some $c >0$ such that there does not exist a family of $r$-uniform hypergraphs with Tur\'an density in $(\alpha, \alpha + c)$. Erd\"os conjectured that all $\alpha\in [0, 1)$ are jumps for any $r$. This was disproven by Frankl and R\"odl when they provided examples of non-jumps. In this paper, we provide a method for finding non-jumps for $r = 3$ using patterns. As a direct consequence, we find a few more examples of non-jumps for $r = 3$.
\end{abstract}


\end{frontmatter}

\section{Introduction}

An $r$-uniform hypergraph is a pair $(V, E)$ where $V$ is a set of vertices and $E$ is a set of $r$-element subsets of $V$ called edges. For convenience, we use $r$-graph in place of $r$-uniform hypergraph here. For an $r$-graph $G$, let $V(G)$ and $E(G)$ denote the set of vertices and edges of $G$, respectively, and let $v(G) = |V(G)|$ and $e(G) = |E(G)|$. For a positive integer $n$, we write $[n]=\{1,\dots,n\}$, and denote by $\binom{[n]}{k}$ the collection of all $k$-element subsets of $[n]$. A basic parameter of an $r$-graph is its density.

\begin{defn}
    The density of an $r$-graph $G$ is
    \[d(G) = \frac{e(G)}{\binom{v(G)}{r}}.\]
\end{defn}

For an $r$-graph $G$ and a positive integer $n$, the extremal number $\ex(n, G)$ denotes the maximum number of edges in an $r$-graph on $n$ vertices that does not contain $G$ as a subgraph. The extremal number for a family of $r$-graphs is defined analogously, except we forbid subgraphs isomorphic to any member of the family. The Tur\'an density of a family describes the limiting behavior of the extremal number for a family of $r$-graphs.

\begin{defn}
    The Tur\'an density of a family of $r$-graphs $\mathcal{F}$ is defined as 
    \[\pi(\mathcal{F}) = \lim_{n\to \infty} \frac{\ex(n, \mathcal{F})}{\binom{n}{r}}.\]
\end{defn}

Such a limit exists by an averaging argument. Jumps describe gaps in the set of Turán densities of all families of $r$-graphs.

\begin{defn}  
    We say that $\alpha\in [0, 1)$ is a jump for $r$ if there exists some $c > 0$, depending on $\alpha$ and $r$, such that there does not exist any family of $r$-graphs $\mathcal{F}$ with $\pi(\mathcal{F}) \in (\alpha, \alpha + c)$. 
\end{defn}

Oftentimes, a different equivalent definition of hypergraph jumps is used in place of this one. The proof of equivalence follows from the argument in \cite{frankrodl1984}.

\begin{prop}
    The following are equivalent: 
    \begin{enumerate}
        \item $\alpha$ is a jump for $r$.
        \item there exists some $c > 0$, depending on $\alpha$ and $r$, such that given any $\epsilon > 0$ and any integer $m\geq r$ there is some integer $N > 0$, depending on $\alpha$, $r$, $m$, and $\epsilon$, such that any $r$-graph on $n\geq N$ vertices with at least $(\alpha + \epsilon)\binom{n}{r}$ edges contains some subgraph on $m$ vertices with at least $(\alpha + c)\binom{m}{r}$ edges.
    \end{enumerate}
\end{prop}

By the Erd\"os-Stone Theorem and another result from Erd\"os \cite{erdos1964},  all $\alpha \in [0, 1)$ are jumps for $r = 2$. Furthermore, Erd\"os' work in \cite{erdos1964} shows that all $\alpha \in \left.\left[0, \frac{r!}{r^r}\right.\right)$ are jumps for $r$. In general, little is known about behavior in the interval $\left.\left[\frac{r!}{r^r}, 1\right.\right)$. A conjecture of Erd\"os \cite{franklpeng2007} says that $\frac{r!}{r^r}$ is a jump for $r$. This problem remains open, but progress has been made in its direction. Baber and Talbot \cite{babertalbot2011} showed that $\alpha\in [0.2299, 0.2316)\cup [0.2871, 8/27)$ are jumps for $r = 3$. Aside from the interval $\left.\left[0, \frac{r!}{r^r}\right.\right)$, these are the only known jumps for $r\geq 3$. However, slightly more can be said about non-jumps.

Frankl and R\"odl \cite{frankrodl1984} were the first to find examples of non-jumps. They showed that for $r\geq 3$ and $\ell > 2r$ the density $1 - \frac{1}{\ell^{r - 1}}$ is not a jump for $r$. Since then, a number of non-jumps and sequences of non-jumps have been found \cite{peng2008, peng2009, shaw2025, yanpeng2023}. Peng \cite{peng2009} showed that if $\alpha\cdot \frac{r!}{r^r}$ is a non-jump for $r \geq 3$, then $\alpha\cdot \frac{p!}{p^p}$ is a non-jump for any $p \geq r$. Recently, Shaw \cite{shaw2025} showed that the smallest non-jumps that can be found using the Frankl-R\"odl method are $\frac{6}{121}\left(5\sqrt{5} - 2\right)$ for $r = 3$ and $2\cdot\frac{r!}{r^r}$ for $r\geq 4$.

In this paper, we provide a way to find non-jumps using patterns. This follows from Shaw's work \cite{shaw2025} on Frankl and R\"odl's method of finding non-jumps. As a result, we find a few examples of non-jumps. Specifically, we show the following.

\begin{thm} \label{thm:nonjump1}
    The density $64/81$ is not a jump for $r = 3$.
\end{thm}

\begin{thm} \label{thm:nonjump2}
    Let $n\in \mathbb{N}$ and $k = \sqrt{3n - 2}$ (not necessarily integral). Then, $1 - \left(\frac{k}{n + k}\right)^2$ is not a jump for $r = 3$.
\end{thm}

\section{Blow-ups and Lagrangians}

Our proofs rely heavily on the ideas of hypergraph blow-ups and Lagrangians. We refer to the notation used in Keevash's survey on the topic \cite{keevash2011}. 

\begin{defn}
Let $G$ be an $r$-graph on the vertices $v_1, \dots, v_n$, and let $t = (t_1, \dots, t_n)\in \mathbb{N}^n$. A \emph{$t$-blow-up} of $G$ is the $r$-graph $G(t)$ obtained as follows: replace each $v_i$ by a set $V_i$ of $t_i$ copies. For each edge $\{v_{i_1},\dots,v_{i_r}\}\in E(G)$, include in $G(t)$ all $r$-sets $\{x_1,\dots,x_r\}$ with $x_j\in V_{i_j}$ for every $j\in[r]$.
\end{defn}

We use $(v, k)$ to refer to the $k$th copy of vertex $v$ in $G(t)$. So each vertex $v_i$ of $G$ is replaced by $V_i = \{(v_i, 1), \dots, (v_i, t_i)\}$. For example, a $(2, 2)$-blow-up of a single edge on two vertices $v$ and $w$ is a square with vertices $(v, 1), (w, 1), (v, 2), (w, 2)$ in cyclic order.

Let $p_G(t)$ denote the number of edges in $G(t)$. Then, $p_G(t) = \sum_{e\in E(G)} \prod_{i\in e} t_i$. Suppose we want to find the largest possible density for $G(t)$ as $|t|\to \infty$, where $|t| = \sum_{i\in [n]} t_i$. For any $t$, 
\[\lim_{m\to \infty} d(G(mt)) = \lim_{m\to \infty} \frac{1}{\binom{m|t|}{r}}p_G(mt) = r!p_G(t_1/|t|, \dots, t_n/|t|).\]
So it suffices to maximize $r!p_G(x)$ over all $x$ in the standard simplex $S = \{x : x_1 + \dots + x_n = 1, x_i \geq 0\quad \forall i\in [n]\}$. Since this sort of optimization appears frequently, a name is given to the essential component to be maximized. 
\begin{defn}
    The Lagrangian of an $r$-graph $G$ is defined as 
    \[\lambda(G) = \max_{x\in S} p_G(x).\]
\end{defn}
There is also a name for the maximum density. 
\begin{defn}
    The blow-up density of an $r$-graph $G$ is defined as
    \[b(G) = r!\lambda(G).\]
\end{defn}
We can think of the $x_i$s of $x$ as optimal weights in $[0, 1]$ that give the whole $r$-graph $G$ a weight $w(G) = \lambda(G)$. 

\section{Patterns}

To guide our search for non-jumps, we will need patterns as defined by Hou, Li, Yang, and Zhang in \cite{hou2024}. A pattern follows all of the same rules as a hypergraph except edges are multisets of vertices (i.e. repetitions allowed). 

\begin{defn}
    An $r$-pattern $P$ is a pair $(n, \mathcal{E})$, where $n\in \mathbb{N}$ is an integer representing the number of vertices of $P$ and $\mathcal{E}$ is a collection of $r$-multisets on $[n]$ representing the edges of $P$. The notation $\mathcal{E}(P)$ refers to the edge set of pattern $P$.
\end{defn}

For example, a $3$-pattern on $3$ vertices may have the edge set $\mathcal{E} = \{112, 123, 223\}$. A blow-up of a pattern is defined analogously to a blow-up of a hypergraph. Let $P$ be an $r$-pattern on vertices $v_1, \dots, v_n$, and let $t = (t_1, \dots, t_n)\in \mathbb{N}^n$. The $t$-blow-up of $P$ is the $r$-pattern $P(t)$ obtained by replacing each vertex $v_i$ with a set $V_i$ of $t_i$ copies. An $r$-multiset $e$ of vertices in $P(t)$ is an edge if the $r$-multiset obtained from $e$ by replacing each vertex of $V_i$ by $v_i$ is an edge of $P$. A simple blow-up of $P$, denoted $P[t]$, is the blow-up $P(t)$ with all edges containing repeated copies removed. For example, if $P = (3, \mathcal{E})$ and $t = (2, 2, 1)$ then $P[t]$ has edge set 
\[
\begin{aligned}
\mathcal{E}(P[t]) = \{&
(1,1)(1,2)(2,1),\ (1,1)(1,2)(2,2),\\
&(1,1)(2,1)3,\ (1,2)(2,1)3,\\
&(1,1)(2,2)3,\ (1,2)(2,2)3,\ (2,1)(2,2)3
\}.
\end{aligned}
\]

The Lagrangian for a pattern is also defined analogously to the Lagrangian for a hypergraph. Let $m_e(i)$ denote the multiplicity of vertex $i$ in edge $e$. Then for any pattern $P$ and $n$-tuple $t = (t_1, \dots, t_n)$,
\[
|\mathcal{E}(P[t])| = \sum_{e\in \mathcal{E}(P)} \prod_{i=1}^n \binom{t_i}{m_e(i)}.
\]
As $m\to\infty$,
\[
|\mathcal{E}(P[mt])| \sim \sum_{e\in \mathcal{E}(P)} \prod_{i=1}^n \frac{(mt_i)^{m_e(i)}}{m_e(i)!} = m^r \sum_{e\in \mathcal{E}(P)} \prod_{i=1}^n \frac{t_i^{m_e(i)}}{m_e(i)!}.
\]
Let
\[p_P(t) = \sum_{e\in \mathcal{E}(P)} \prod_{i=1}^n \frac{t_i^{m_e(i)}}{m_e(i)!}.
\]
Then
\[
\lim_{m\to\infty} d(P[mt]) = \lim_{m\to\infty} \frac{|\mathcal{E}(P[mt])|}{\binom{m|t|}{r}} = \frac{r!p_P(t)}{|t|^r} = r!p_P\left(\frac{t_1}{|t|},\dots,\frac{t_n}{|t|}\right).
\]
As before, if we wish to maximize the asymptotic density of $P[t]$, then it suffices to maximize $r!p_P(x)$ over the standard simplex $S$. Furthermore, the Lagrangian (or weight) of $P$ is defined as $\lambda(P) = \max_{x\in S} p_P(x)$, and its blow-up density is defined as $b(P) = r!\lambda(P)$. 

\section{Finding Non-jumps}
 
Now, we outline the strategy for finding non-jumps, building on the work of Shaw \cite{shaw2025}. First, we introduce the Frankl-R\"odl construction $\FR_v(P)$ from a pattern $P$.

\begin{defn}
    Suppose we are given an $r$-pattern $P = (n, \mathcal{E})$ and $v\in [n]$. Let $t$ be defined such that $t_v = r$ and $t_i = 1$ for all other vertices $i\in [n]$. Let $P'$ be an $r$-pattern on the vertices of $P(t)$ with 
    \[\mathcal{E}(P') = \{e\in \mathcal{E}(P(t)) :  m_e((v, i)) \leq 1 \quad \forall i\in \{2, \dots, r\}\},\] 
    Then, $\FR_v(P)$ is defined as
    \[\FR_v(P) = P'\cup \{(v, 1)\cdots(v, r)\}.\]
\end{defn}

This construction is useful for providing a sufficient condition for $\alpha$ to be a non-jump for $r$. The proof of this can be found in \cite{shaw2025}.

\begin{thm} \label{thm:frmethod}
    Let $P$ be an $r$-pattern, and $v$ be a vertex of $P$ such that an optimal weighting of $P$ assigns $v$ positive weight. Suppose
    \[\lambda(\FR_v(P)) = \lambda(P) < 1.\] 
    Then, $r!\lambda(P)$ is not a jump for $r$-graphs.
\end{thm}

As the name suggests, this method was first used by Frankl and R\"odl \cite{frankrodl1984} to show the existence of non-jumps. Using this technique, we prove the densities given in Theorem \ref{thm:nonjump1} and Theorem \ref{thm:nonjump2} are non-jumps. First, we need a lemma for patterns with equivalent vertices. Two vertices in a pattern are deemed equivalent if their labelings can be swapped to produce a pattern isomorphic to the original. We use the notation $w_v$ for the weight of vertex $v$.

\begin{lem} \label{lem:equiv}
    Let $P = (n, \mathcal{E})$ be a $3$-pattern, and let $i$ and $j$ be vertices in $P$. If $i$ and $j$ are equivalent, then there exists an optimal weighting such that either $w_i = w_j$ or one of $w_i$ and $w_j$ is $0$.
\end{lem}
\begin{proof}
    Suppose the weights of all vertices in $[n]\setminus\{i,j\}$ are fixed. Since the total weight is fixed, the sum $s = w_i + w_j$ is also fixed. Then the contribution of $i$ and $j$ to $w(P)$ can be written as
    \[
    w(P) = (w_i + w_j)C_1 + (w_i^2 + w_j^2)C_2 + w_iw_j\,C_3
    \]
    where $C_1,C_2,C_3$ depend only on the weights of the remaining vertices. Indeed, terms such as $w_i^3 + w_j^3$ and $w_i^2w_j + w_iw_j^2$ can be rewritten using 
    \begin{align*}
        w_i^3 + w_j^3 &= s(w_i^2 + w_j^2) - sw_iw_j \\
        w_i^2w_j + w_iw_j^2 &= sw_iw_j.
    \end{align*}
    If both $w_i$ and $w_j$ are positive for an optimal weighting, then by the
    Lagrange multiplier method,
    \[
    \frac{\partial}{\partial w_i} w(P) = \frac{\partial}{\partial w_j} w(P).
    \]
    This simplifies to
    \[
    (2C_2 - C_3)(w_i - w_j)=0.
    \]
    Therefore, either $w_i = w_j$ or $C_3 = 2C_2$. In the second case,
    \[
    w(P) = sC_1 + s^2C_2,
    \]
    so the value of $w(P)$ does not change with the relative weights of $w_i$ and $w_j$. Consequently, there exists an optimal weighting with $w_i = w_j$. Otherwise, an optimal weighting can be chosen on the boundary of the feasible region, so one of $w_i$ or $w_j$ is $0$.
\end{proof}

We also need the blow-up density of a specific small $3$-pattern.
\begin{lem} \label{lem:smallpattern}
    Let $P = \{112, 122\}$. Then, 
    \[\lambda(\FR_1(P)) = \lambda(P) = \frac{1}{8}.\]
\end{lem}
\begin{proof}
    Let $w$ be an optimal weighting of $P$ and $w_1$ and $w_2$ be the weights of the vertices. Since both $1$ and $2$ are equivalent, by Lemma \ref{lem:equiv} either $w_1 = w_2$ or one of $w_1$ or $w_2$ is $0$. Since we cannot have the second case, $w_1 = w_2 = \frac{1}{2}$. So 
    \[\lambda(P) = \frac{w_1^2w_2}{2} + \frac{w_1w_2^2}{2} = \frac{1}{8}.\]
    
    Now, let $w'$ be an optimal weighting of $\FR_1(P)$ and $w_{(1, 1)}'$, $w_{(1, 2)}'$, $w_{(1, 3)}'$, and $w_2'$ be the weights of the vertices. By Lemma \ref{lem:equiv}, either $w_{(1, 2)}' = w_{(1, 3)}'$ or exactly one of $w_{(1, 2)}'$ or $w_{(1, 3)}'$ is $0$. In the second case, changing both weights to equal half of the nonzero one does not decrease $w'(\FR_1(P))$, so we assume $w_{(1, 2)}' = w_{(1, 3)}'$. Let $a = w_{(1, 1)}'$, $b = 2w_{(1, 2)}' = w_{(1, 2)}' + w_{(1, 3)}'$, and $c = w_2'$. Then, 
    \[w'(\FR_1(P)) = \frac{ab^2}{4} + \frac{a^2c}{2} + \frac{b^2c}{4} + abc + \frac{(a + b)c^2}{2}.\]
    We want to maximize $w'(\FR_1(P))$ subject to the constraint $a + b + c = 1$. Assume $a$, $b$, and $c$ are all positive. Then, using the Lagrange multiplier method,
    \[\frac{\partial}{\partial a} w'(\FR_1(P)) = \frac{\partial}{\partial b} w'(\FR_1(P)) = \frac{\partial}{\partial c} w'(\FR_1(P)).\]
    We have 
    \begin{align}
        \frac{\partial}{\partial a} w'(\FR_1(P)) &= \frac{b^2}{4} + ac + bc + \frac{c^2}{2} \label{eq:1}\\
        \frac{\partial}{\partial b} w'(\FR_1(P)) &= \frac{ab}{2} + \frac{bc}{2} + ac + \frac{c^2}{2} \label{eq:2}\\
        \frac{\partial}{\partial c} w'(\FR_1(P)) &= \frac{a^2}{2} + \frac{b^2}{4} + ab + (a + b)c \label{eq:3}.
    \end{align}
    Subtracting \eqref{eq:1} from \eqref{eq:3} gives us $b = \frac{c^2 - a^2}{2a}$, and subtracting \eqref{eq:1} from \eqref{eq:2} gives us $b = 2(a - c)$. If $c < a$, we get $b < 0$ from the first equality, and if $c > a$ we get $b < 0$ from the last equality. Finally, if $a = c$, then $b = 0$. This contradicts the assumption $b > 0$; therefore, no interior critical point exists, and the optimum occurs on the boundary. If $a = 0$, then the optimal weight is 
    \[\frac{b^2c}{4} + \frac{bc^2}{2}\leq \frac{b^2c}{2} + \frac{bc^2}{2}.\] 
    Note that the right expression is the weight of $P$ for $w_1 = b$ and $w_2 = c$. So the optimal weight is at most $\lambda(P) = 1/8$ in this case. Likewise, if $c = 0$, then the optimal weight is $\frac{ab^2}{4} \leq \frac{1}{27} < \frac{1}{8}$. If $b = 0$, then we are left with the same expression as that for $w(P)$. So $\lambda(\FR_1(P)) = \lambda(P) = \frac{1}{8}$. 
\end{proof} 

So by Theorem \ref{thm:frmethod}, $\frac{3}{4}$ is a non-jump for $r = 3$. Now, we state and prove our main result.

\begin{thm} \label{thm:main}
    Let $P = (n, \mathcal{E})$ be a $3$-pattern such that $111\notin \mathcal{E}$ and an optimal weighting of $P$ assigns vertex $1$ positive weight. Suppose 
    \[\{122\}\cup \{11i: i\in [n]\setminus \{1\}\}\subseteq \mathcal{E}.\] 
    Then, $\lambda(\FR_1(P)) = \lambda(P)$.
\end{thm}

\begin{proof}
    Let $w$ be an optimal weighting for $\FR_1(P)$. Without loss of generality, $w_{(1, 2)} = w_{(1, 3)}$ by the same argument from Lemma \ref{lem:smallpattern}. Let $a = w_{(1, 1)}$, $b = w_{(1, 2)} + w_{(1, 3)} = 2w_{(1, 2)}$, and $c = \sum_{i = 2}^n w_i$. Then,
    \begin{align*}
        w(\FR_1(P)) &= \frac{ab^2}{4} + \frac{a^2c}{2} + \frac{b^2c}{4} + abc + \frac{1}{2}(a + b)\left(\sum_{\substack{1ii\in \mathcal{E} \\ i\in [n]\setminus \{1\}}} w_i^2\right) + (a + b)\left(\sum_{\substack{1ij\in \mathcal{E} \\ 2\leq i < j\leq n}} w_iw_j\right) \\
        &+ \sum_{\substack{ijk\in \mathcal{E} \\ 2\leq i < j < k \leq n}} w_iw_jw_k + \sum_{\substack{iij\in\mathcal{E} \\ i,j\in [n]\setminus\{1\}\\ i\neq j}} \frac{w_i^2w_j}{2} + \sum_{\substack{iii\in \mathcal{E} \\ i\in [n]\setminus \{1\}}} \frac{w_i^3}{6}.
    \end{align*}
    Fix $\alpha_2' \coloneq \frac{w_2}{c}, \dots, \alpha_n' \coloneq \frac{w_n}{c}$. (If $c = 0$, then $w(\FR_1(P))\leq 1/27 < 1/8$, which is suboptimal by Lemma \ref{lem:smallpattern}. So we may assume $c > 0$.) Then, let $\alpha_1$ and $\alpha_2$ be constants, depending only on the normalized weights $\alpha_i'$, such that 
    \[w(\FR_1(P)) = \frac{ab^2}{4} + \frac{a^2c}{2} + \frac{b^2c}{4} + abc + \frac{1}{2}(a + b)\alpha_1c^2 + \frac{1}{6}\alpha_2c^3.\]
    We have 
    \[\frac{1}{2}(a + b)\left(\sum_{1ii\in \mathcal{E}} w_i^2\right) + (a + b)\left(\sum_{\substack{1ij\in \mathcal{E} \\ 2\leq i < j\leq n}} w_iw_j\right) \leq \frac{1}{2}(a + b)c^2.\]
    Likewise, 
    \[\sum_{\substack{ijk\in \mathcal{E} \\ 2\leq i < j < k \leq n}} w_iw_jw_k + \sum_{\substack{iij\in\mathcal{E} \\ i,j\in [n]\setminus\{1\}\\ i\neq j}} \frac{w_i^2w_j}{2} + \sum_{\substack{iii\in \mathcal{E} \\ i\in [n]\setminus \{1\}}} \frac{w_i^3}{6} \leq \frac{1}{6}c^3.\]
    Therefore, $\alpha_1, \alpha_2\in [0, 1]$. Let $s = 1 - c$, and define $x\in [0, 1]$ such that $a = xs$ and $b = (1 - x)s$. Then, 
    \[w(\FR_1(P)) = \frac{s^3}{4}x(1 - x)^2 + \frac{s^2c}{4}(1 + 2x - x^2) + \frac{1}{2}\alpha_1sc^2 + \frac{1}{6}\alpha_2c^3.\]
    We want to show that $x = 1$ (or $b = 0$) for an optimal weighting, since this implies that $\lambda(\FR_1(P)) = \lambda(P)$. Fix any $c\in [0, 1]$. Then, 
    \[\frac{d}{dx}w(\FR_1(P)) = \frac{s^2}{4}(1 - x)(s(1 - 3x) + 2c).\]
    If $\frac{d}{dx}w(\FR_1(P)) = 0$ for $x\in (0, 1)$, then $x = x^* \coloneq \frac{1}{3}\left(\frac{2c}{s} + 1\right)$. If $c \geq \frac{1}{2}$, then $x^* \geq 1$, which implies there is no critical point in $(0, 1)$. If $c \geq \frac{1}{2}$, then $\frac{d}{dx}w(\FR_1(P)) \geq 0$, which implies $x = 1$ maximizes $w(\FR_1(P))$. On the other hand, if $c < \frac{1}{2}$, then $x = x^*$ maximizes $w(\FR_1(P))$, since $\frac{d}{dx}w(\FR_1(P))$ is positive for $x < x^*$ and negative for $x > x^*$. Therefore, it suffices to show $c \geq \frac{1}{2}$ for an optimal weighting. 
    
    We now compare the boundary solution $x = 1$ with the interior critical point $x = x^*$ and show that the latter never yields a larger value. Define $B : [0, 1]\to [0, 1]$ to be the function of $c$ obtained from setting $x = 1$ in $w(\FR_1(P))$, and define $I : [0, 1]\to [0, 1]$ to be the function of $c$ obtained from setting $x = x^*$ in $w(\FR_1(P))$. That is, 
    \begin{align*}
        B(c) &= \frac{s^2c}{2} + \frac{1}{2}\alpha_1sc^2 + \frac{1}{6}\alpha_2c^3 \\
        I(c) &= \frac{1}{27} + \frac{5c}{18} - \frac{5c^2}{9} + \frac{11c^3}{54} + \frac{1}{2}\alpha_1sc^2 + \frac{1}{6}\alpha_2c^3.
    \end{align*}
    Then, the optimal weighting is 
    \[w(\FR_1(P)) = \max\left\{\max_{c\in [0, 1/2]} I(c), \max_{c\in [1/2, 1]} B(c)\right\}.\]
    So it suffices to show that $I(c)$ does not provide an optimal weighting. We do this by assuming that $I(c)$ is optimal and deriving the contradiction $B(1/2) > I(c)$. Let $Q$ be a quadratic polynomial in $c$ defined by
    \[Q(c; \alpha_1, \alpha_2) = (-108c^2 + 54c + 27)\alpha_1 + (36c^2 + 18c + 9)\alpha_2 + 44c^2 - 98c + 11.\]
    Then, $B(1/2) - I(c)$ can be factored as 
    \[B(1/2) - I(c) = \frac{(1 - 2c)Q(c; \alpha_1, \alpha_2)}{432}.\]
    Since $1 - 2c > 0$ for $c < \frac{1}{2}$, we want to show for any fixed $c < \frac{1}{2}$, 
    \[\min Q(c; \alpha_1, \alpha_2) > 0,\]
    where the minimum is taken over all admissible choices $(\alpha_1, \alpha_2)$. Define 
    \begin{align*}
        t(c) &= \frac{19}{216} - \frac{5c}{18} + \frac{5c^2}{9} - \frac{11c^3}{54},\\
        u(c) &= \frac{1}{2}sc^2, \\
        v(c) &= \frac{1}{6}c^3.
    \end{align*}
    If the interior solution $x = x^*$ is optimal, then by Lemma \ref{lem:smallpattern} we have $I(c) \geq 1/8$, which implies
    \[u(c)\alpha_1 + v(c)\alpha_2 \geq t(c).\]
    Therefore, $(\alpha_1, \alpha_2)$ lies in the feasible set
    \[S\coloneq \{(\alpha_1, \alpha_2)\in [0, 1]^2 : u(c)\alpha_1 + v(c)\alpha_2 \geq t(c)\}.\]
    Since $Q$ is linear in $\alpha_1$ and $\alpha_2$, for fixed $c$ its minimum over $S$ (when $S$ is nonempty) is attained at a vertex of $S$. For $c < \frac{1}{2}$, 
    \[t(c) - u(c) = \frac{(1 - 2c)^2(16c + 19)}{216} > 0.\]
    Furthermore, $u(c) + v(c) \geq t(c)$, since $S\neq \emptyset$. So the constraint line intersects $\alpha_1 = 1$ between $(1, 0)$ and $(1, 1)$, and therefore intersects the boundary of $[0, 1]^2$ at $(\alpha_1, \alpha_2) = \left(1, \frac{t(c) - u(c)}{v(c)}\right)$. Define
    \begin{align*}
        q_0(c) &= 44c^2 - 98c + 11, \\
        q_1(c) &= -108c^2 + 54c + 27, \\
        q_2(c) &= 36c^2 + 18c + 9.
    \end{align*}
    Then, 
    \begin{equation} \label{eq:slope}
        q_1(c) - \frac{q_2(c)u(c)}{v(c)} = -\frac{27}{c} < 0.
    \end{equation}
    Let $(\alpha_1, \alpha_2)$ be a point on the constraint line $u(c)\alpha_1 + v(c)\alpha_2 = t(c)$. Then, $\alpha_2 = \frac{t(c) - u(c)\alpha_1}{v(c)}$. If we plug that value into $Q$, we get 
    \begin{align*}
        Q(c; \alpha_1, \alpha_2) &= q_0(c) + q_1(c)\alpha_1 + q_2(c)\alpha_2 \\ 
        &= \left(q_0(c) + \frac{q_2(c)t(c)}{v(c)}\right) + \alpha_1\left(q_1(c) - \frac{q_2(c)u(c)}{v(c)}\right).
    \end{align*}
    By inequality (\ref{eq:slope}), when restricted to the constraint line $u(c)\alpha_1 + v(c)\alpha_2 = t(c)$, $Q$ is decreasing in $\alpha_1$. Therefore, it is minimized at $\alpha_1 = 1$. We plug $(\alpha_1, \alpha_2) = \left(1, \frac{t(c) - u(c)}{v(c)}\right)$ into $Q$ and simplify to get
    \[Q\left(c; 1, \frac{t(c) - u(c)}{v(c)}\right) = \frac{(1 - 2c)(16c + 19)}{4c^3}.\]
    For $c < \frac{1}{2}$, the expression is positive, which implies $Q(c; \alpha_1, \alpha_2) > 0$ and therefore $B(1/2) > I(c)$ for all admissible $(\alpha_1, \alpha_2)$ and all $c < \frac{1}{2}$, as desired. So any optimal weighting must satisfy $c\geq 1/2$. In this interval, $x = 1$ (or equivalently $b = 0$) is optimal. Therefore an optimal weighting of $\FR_1(P)$ corresponds to a weighting of $P$, giving $\lambda(\FR_1(P))\leq \lambda(P)$. The reverse inequality follows by assigning weight $0$ to $(1,2)$ and $(1,3)$ in $\FR_1(P)$. Therefore, $\lambda(\FR_1(P))=\lambda(P)$.
\end{proof}

\section{Proof of Theorem \ref{thm:nonjump1}}
Let $P = \{123, 122, 112, 113, 223\}$. By Theorem \ref{thm:frmethod} and Theorem \ref{thm:main}, it suffices to show that vertex $1$ receives a positive weight in an optimal weighting for $P$ and
\[\lambda(P) = \frac{32}{243}.\]
This will show us that $3!\cdot \frac{32}{243} = \frac{64}{81}$ is not a jump for $r = 3$. Let $w$ be an optimal weighting for $P$. Then,  
\[w(P) = w_1w_2w_3 + \frac{w_1w_2^2}{2} + \frac{w_1^2w_2}{2} + \frac{w_1^2w_3}{2} + \frac{w_2^2w_3}{2}.\]
Suppose $w_1$, $w_2$, and $w_3$ are positive. Then, by Lemma \ref{lem:equiv} we have $w_1 = w_2$, which implies
\[w(P) = 2w_1^2w_3 + w_1^3.\] 
We want to maximize $w(P)$ subject to the constraint 
\[2w_1 + w_3 = 1.\] 
Using the method of Lagrange multipliers,
\begin{align*}
    \frac{\partial}{\partial w_1} w(P) &= 2\frac{\partial}{\partial w_3} w(P) \\
    4w_1w_3 + 3w_1^2 &= 4w_1^2 \\
    w_1 - 4w_3 &= 0.
\end{align*}
Then, using the constraint, we get $w_1 = w_2 = 4/9$ and $w_3 = 1/9$, leaving $w(P) = 32/243$. Otherwise, an optimal weighting exists on the boundary. If $w_1 = 0$, then $w_2 = \frac{2}{3}$ and $w_3 = \frac{1}{3}$ maximizes $w(P)$. This leaves $w(P) = 1/27 < 32/243$. We get the same result when $w_2 = 0$. If $w_3 = 0$, then by Lemma \ref{lem:smallpattern} we have $w(P) \leq 1/8 < 32/243$. So $\lambda(P) = 32/243$. 

\section{Proof of Theorem \ref{thm:nonjump2}}
Let $n$ be fixed, and define $P$ to be the pattern on $n + 1$ vertices with edge set
\[\binom{[n + 1]}{3}\cup \{1ij : i, j\in [n + 1]\setminus \{1\}\}\cup \{11i : i\in [n+1]\setminus\{1\}\}.\] 
By Theorem \ref{thm:frmethod} and Theorem \ref{thm:main}, it suffices to show that vertex $1$ receives a positive weight in an optimal weighting for $P$ and
\[\lambda(P) = \frac{1}{6}\left(1 - \frac{k^2}{(n + k)^2}\right).\]
By Lemma \ref{lem:equiv}, there exists an optimal weighting $w(P)$ such that all of $w_2, \dots, w_{n + 1}$ are equal, or some of $w_2, \dots, w_{n + 1}$ are $0$ and the rest are equal. Let $m$ be the number of these vertices that are assigned positive weight. Then, define $a = w_1$ and $b = w_2 + \dots + w_{n + 1}$, so
\begin{align*}
    w(P) &= \binom{m}{2}\frac{ab^2}{m^2} + \binom{m}{3}\frac{b^3}{m^3} + \frac{a^2b}{2} + \frac{ab^2}{2m} \\
    &= \frac{ab^2}{2} + \binom{m}{3}\frac{b^3}{m^3} + \frac{a^2b}{2}.
\end{align*} 
Since $\frac{1}{m^3}\binom{m}{3} = \frac{(m - 1)(m - 2)}{6m^2}$ is non-decreasing in $m$, it follows that $w(P)$ is maximized when $m = n$. We want to maximize $w(P)$ subject to the constraint
\[a + b = 1.\] 
First, suppose $a, b > 0$. Then, using the Lagrange multipliers method,
\begin{align*}
    \frac{\partial}{\partial a} w(P) &= \frac{\partial}{\partial b} w(P) \\
    \frac{b^2}{2} + ab &= ab + \binom{n}{3}\frac{3b^2}{n^3} + \frac{a^2}{2} \\
    \left(1 - \frac{(n - 1)(n - 2)}{n^2}\right)b^2 &= a^2 \\
    a &= \frac{k}{n}\cdot b.
\end{align*}
We use the constraint equation to get
\begin{align*}
    b &= \frac{1}{1 + \frac{k}{n}} = \frac{n}{n + k} \\
    a &= \frac{k}{n}\cdot b = \frac{k}{n + k}.
\end{align*}
Finally, plugging these values for $a$ and $b$ into the expression for $w(P)$ we get 
\[w(P) = \frac{1}{6}\left(1 - \frac{k^2}{(n + k)^2}\right).\]
Suppose instead an optimal weighting exists on the boundary. If $a = 0$ and $b = 1$,
\[w(P) =  \frac{1}{n^3}\binom{n}{3} = \frac{(n - 1)(n - 2)}{6n^2} = \frac{1}{6}\left(1 - \frac{k^2}{n^2}\right).\]
However, 
\[\frac{1}{6}\left(1 - \frac{k^2}{n^2}\right) < \frac{1}{6}\left(1 - \frac{k^2}{(n + k)^2}\right).\] 
Therefore, $\lambda(P) = \frac{1}{6}\left(1 - \frac{k^2}{(n + k)^2}\right)$, as desired. 

\section*{Acknowledgments} 
It is a pleasure to thank my mentor, V\'ictor Hugo Almendra Hern\'andez, for his frequent feedback and discussion about my ideas for the paper. I also want to thank Professor Razborov for introducing me to this topic and guiding my study of it. Lastly, I want to thank Professor May for organizing the REU that made this research possible.

\section*{Funding} 
This work was supported by the University of Chicago Mathematics Research Experience for Undergraduates (REU) Program.

\bibliographystyle{elsarticle-num}
\bibliography{refs}

\end{document}